\documentclass[a4paper,12pt]{article}
\usepackage{CJKutf8}
\usepackage{CJK}
\usepackage{amsfonts}
\usepackage{pifont}
\usepackage{bm}
\usepackage{latexsym,amsmath,amssymb,cite,amsthm}
\usepackage{color,eucal,enumerate,mathrsfs}
\usepackage[normalem]{ulem}
\usepackage{amsmath}
\usepackage[colorlinks, linkcolor=red,  anchorcolor=blue,citecolor=blue]{hyperref}

\numberwithin{equation}{section}
\theoremstyle{plain}
\newtheorem{theorem}{Theorem}[section]
\newtheorem{corollary}[theorem]{Corollary}
\newtheorem{lemma}[theorem]{Lemma}
\newtheorem{proposition}[theorem]{Proposition}
\theoremstyle{definition}
\newtheorem{definition}[theorem]{Definition}
\theoremstyle{remark}

\newtheorem{remark}[theorem]{Remark}

\newcommand{\opt}{\rm Opt}
\newcommand{\geo}{\rm Geo}

\newcommand{\lmt}[2]{\mathop{\lim}_{{#1} \rightarrow {#2}} }

\newcommand{\lmts}[2]{\mathop{\overline{\lim}}_{{#1} \rightarrow {#2}} }

\newcommand{\mm}{\mathfrak m}
\newcommand{\ms}{(X,\d,\mm)}

\newcommand{\cdkn}{{\rm CD}(K, N)}

\newcommand{\bekn}{{\rm BE}(K, N)}
\newcommand{\mcp}{{\rm MCP}(K, N)}

\newcommand{\N}{\mathbb{N}}

\newcommand{\R}{\mathbb{R}}

\newcommand{\supp}{\mathop{\rm supp}\nolimits}   
\newcommand{\Lip}{\mathop{\rm Lip}\nolimits}

\renewcommand{\d}{{\mathrm d}}
\newcommand{\dt}{{\d t}}
\newcommand{\ddt}{{\frac \d\dt}}

\newcommand{\D}{{\mathrm D}}
\newcommand{\restr}[1]{\lower3pt\hbox{$|_{#1}$}}

\newcommand{\nchi}{{\raise.3ex\hbox{$\chi$}}}

\title{\large{\bf Sharp  $p$-Poincar\'e inequalities under  measure contraction property}
}

\begin{document}
\author{Bang-Xian Han\thanks{Department of Mathematics, Technion-Israel Institute of Technology, Haifa 32000, Israel. Email: hanbangxian@gmail.com.
}
}

\date{\today} 
\maketitle

%
\begin{abstract}
We obtain sharp  estimate on  $p$-spectral gaps,  or equivalently optimal constant in $p$-Poincar\'e inequalities,    for  metric measure spaces satisfying measure contraction property.   We also   prove the  rigidity for the sharp $p$-spectral gap.
\end{abstract}

\tableofcontents

\section{Introduction}
Sharp estimates on  spectral gap for $p$-Laplacian, or equivalently, the optimal constant in $p$-Poincar\'e inequalities is a classical problem in comparison geometry. It addresses the following basic problem.   Given a family $\mathcal F:=\{(X_\alpha, \d_\alpha, \mm_\alpha): \alpha \in \mathcal A\}$  of metric measure spaces, the corresponding optimal constant $\lambda_{\mathcal F}$  in $p$-Poincar\'e inequalities  is defined  by
\begin{equation}\label{eq1:intro}
\lambda_{\mathcal F}:= \mathop{\inf}_{\alpha \in \mathcal A} \inf \left \{\frac {\int_{X_\alpha}  |{\nabla }_{ \d_\alpha} f|^p \,\d \mm_\alpha}{\int_{X_\alpha} |f|^p  \,\d \mm_\alpha}: f \in \Lip \cap L^p, \int_{X_\alpha} f|f|^{p-2}  \,\d \mm_\alpha=0, f\neq 0 \right \},
 \end{equation}
 where the local Lipschitz constant  $ |{\nabla }_{ \d_\alpha} f|: X_\alpha \mapsto \R$ is defined by
 \[
  |{\nabla }_{ \d_\alpha} f|(x):=\lmts{y}{x} \frac{|f(y)-f(x)|}{\d_{\alpha}(y, x)}.
 \]

One of the most studied  families of metric measure spaces  is  Riemannian manifolds with lower  Ricci curvature bound $K\in \R$, upper dimension bound $N>0$  and diameter bound $D>0$.  In this case,  $\lambda_{\mathcal F}$ is the minimum of all  first positive eigenvalues of the  {$p$}-Laplacian (assuming Neumann boundary conditions if the boundary is not empty).  Based on a refined gradient comparison technique and a careful analysis of the underlying model spaces, sharp estimate on the first eigenvalue of the  {$p$}-Laplacian was finally  obtained by Valtorta and Naber in \cite{Valtorta12, NV14}.

Another important family is weighted Riemannian manifolds (called smooth metric measure spaces) satisfying   $\bekn$ curvature-dimension condition  \`a la Bakry-\'Emery \cite{B-H, BE-D}.  More generally, thanks to the deveploment of optimal transport theory, it was realized that  Bakry-\'Emery's   condition in smooth setting can  be equivalently characterized by  convexity  of an entropy functional along $L^2$-Wasserstein geodesics  (c.f. \cite{CMS01} and  \cite{SVR-T}). In this direction, metric measure spaces satisfying  $\cdkn$ condition  was introduced by Lott-Villani \cite{Lott-Villani09} and Sturm \cite{S-O1, S-O2}.   This class of metric measure spaces with synthetic lower Ricci curvature bound and upper  dimension  bound includes  the previous smooth examples,  and is  closed in  the measured Gromov-Hausdorff topology. Recently, using  measure decomposition technique on Riemannian manifolds developed by Klartag \cite{KlartagNeedle} (and by Cavalletti-Mondino  \cite{CM-ISO} on metric measure spaces),  sharp  $p$-Poincar\'e inequalities under the $\bekn$ condition and  the $\cdkn$ condition have been obtained by E. Calderon in his Ph.D thesis \cite{Calderon18}.

In addition, Measure Contraction Property $\mcp$ was introduced independently by Ohta  \cite{Ohta-MCP} and Sturm  \cite{S-O2} as a weaker variant of $\cdkn$ condition. 
The family $\mcp$ is strictly  larger  than $\cdkn$. It was discovered by Juillet \cite{Juillet09}  that the $n$-th Heisenberg group equipped with the left-invariant measure, which is the simplest sub-Riemannian space, does not satisfy any  $\cdkn$ condition  but do satisfy ${\rm MCP}(0,N)$  for  $N\geq 2n+3$.
More recently, interpolation inequalities \`a la Cordero-Erausquin–McCann–Schmuckenshl\"ager \cite{CMS01} were obtained, under suitable modifications, by Barilari and Rizzi \cite{BarilariRizzi18} in the ideal sub-Riemannian setting, Badreddine and Rifford \cite{BR-MCP} for Lipschitz Carnot group,  and by Balogh, Krist\'aly and Sipos  \cite{BKS18} for the Heisenberg group.  As a consequence, more and more  examples of spaces verifying MCP but not CD have been found, e.g. the generalized H-type groups and the Grushin plane (for more details, see \cite{BarilariRizzi18}).

In \cite{HM-MCP},   the author and E. Milman proved a sharp Poincar\'e inequality for subsets  of (essentially non-branching) $\mcp$  metric measure spaces, whose diameter is bounded from above by $D$. 
The current paper is a subsequent work of  \cite{HM-MCP}.   We will study  the general  $p$-poincar\'e  inequality  within the class of spaces verifying measure contraction property.  Thanks to measure decomposition theorem (c.f. Theorem 3.5 \cite{CM-Laplacian}), it suffices to study the corresponding eigenvalue problems on  one-dimensional model spaces introduced by E. Milman \cite{Milman-JEMS}.  In particular, we identify a family of one-dimensional $\mcp$-densities with diameter $D$, not verifying $\cdkn$, achieving the optimal constant $\lambda^p_{K, N, D}$.

As a basic problem in metric geometry, the rigidity theorem helps us to understand more about the spaces under study.    For the family of metric measure spaces satisfying ${\rm RCD}(K, N)$  condition with $K>0$,  a  space that reaches  the equality  in  \eqref{eq1:intro} must have maximal diameter $\pi \sqrt{\frac{N-1}K}$.  By maximal diameter theorem this space  is isomorphic  to a spherical suspension (see \cite{CM-S}  and references therein for details).   For  $\mcp$ spaces,  the situation is very different. For $K > 0$, due to lack of monotonicity, we do not know whether a space that  reaches the minimal spectrum has maximal diameter. For  $K > 0$,  by monotonicity (Proposition \ref{lemma1}) and  one-dimensional rigidity (Theorem \ref{th-rigidity})    we can  prove the   rigidity  Theorem \ref{th2}.

\bigskip

\noindent \textbf{Acknowledgement}:  This research is part of a project which has received funding from the European Research Council (ERC) under the European Union's Horizon 2020 research and innovation programme (grant agreement No. 637851).  The author thanks  Emanuel Milman for helpful discussions and comments.

\section{Prerequisites}
Let $(X, \d)$ be a complete metric space and $\mm$ be a locally finite Borel measure with full support. Denote by $\geo(X,\d)$ the space of geodesics. We say that a set $\Gamma \subset  \geo(X,\d)$  is non-branching  if  for any $\gamma^1, \gamma^2 \in \Gamma$, it holds:
\[
\exists t \in (0,1) ~~\text{s.t.}~ ~ \gamma_s^1 =\gamma_s^2,~\forall  s\in [0, t] \Rightarrow  \gamma_s^1 =\gamma_s^2, ~ \forall s\in [0,1].
\]

 Let $(\mu_t) $ be a $L^2$-Wasserstein geodesic.   Denote by $\opt\geo(\mu_0, \mu_1)$ the space of all   probability measures $\Pi \in    \mathcal{P}(\geo(X,\d))$  such that  $(e_t)_\sharp \Pi=\mu_t$ (c.f.  Theorem 2.10 \cite{AG-U}) where $e_t$ denotes the evaluation map $e_t(\gamma):=\gamma_t$.  We say that  $\ms$ is  essentially non-branching if for any $\mu_0, \mu_1 \ll \mm$,  any  $\Pi \in  \opt\geo(\mu_0, \mu_1)$ is  concentrated on a set of non-branching geodesics.

It is clear that if $(X, \d)$ is a smooth Riemannian manifold then any subset $\Gamma \subset  \geo(X,\d)$ is a set of non-branching geodesics, in particular any smooth Riemannian manifold is essentially non-branching. In addition,  many sub-Riemannian spaces are also  essentially non-branching, which follows from the existence and uniqueness of the optimal transport map on some ideal sub-Riemannian manifolds (c.f. \cite{FigalliRifford10}).

Given $K, N\in \R$, with $N > 1$, we set for $(t, \theta) \in [0, 1] \times \R^+$,
\[
 \sigma^{(t)}_{K, N}   \big(\theta):=\left\{\begin{array}{llll}
\infty, &\text{if}~~ K\theta^2 \geq (N-1)\pi^2,\\
\frac{\sin(t\theta \sqrt{K/(N-1)})}{\sin(\theta \sqrt{K/(N-1)})}, &\text{if}~~ 0<K\theta^2<(N-1)\pi^2,\\
t, &\text{if}~~K\theta^2=0,\\
\frac{\sinh(t\theta \sqrt{-K/(N-1)})}{\sinh(\theta \sqrt{-K/(N-1)})}, &\text{if} ~~ K\theta^2<0.
\end{array}
\right.
\]
and
\[
 \tau^{(t)}_{K, N} :=t^{\frac 1N} \Big ( \sigma^{(t)}_{K, N-1}     \Big )^{1-\frac1N}.
\]
\begin{definition}[Measure Contraction Property $\mcp$]
We say that an essentially non-branching metric measure space $\ms$ satisfies measure contraction property $\mcp$ if for any point  $o\in X$ and Borel set $A\subset X$ with $0<\mm(A)<\infty$ (and with $A \subset B(o, \sqrt{(N-1)/K}$ if $K>0$),   there is $\Pi \in \opt\geo(\frac1{\mm(A)}\mm\restr{A}, \delta_o)$ such that the following inequality holds for all $t\in [0,1]$
\begin{equation}\label{eq-MCP}
\frac1{\mm(A)} \mm  \geq (e_t)_\sharp \big[ \tau^{(1-t)}_{K, N}  \big(\d(\gamma_0, \gamma_1)\big)^N\, \Pi(\d \gamma) \big].
\end{equation}
\end{definition}

\begin{theorem}[Localization for $\mcp$ spaces, Theorem 3.5 \cite{CM-Laplacian}]\label{prop-l1}
Let $\ms$ be an essentially non-branching metric measure space  satisfying $\mcp$ condition for some $K\in \R$ and $N \in (1, \infty)$. Then
for any 1-Lipschitz function $u$ on $X$, the non-branching transport set $\mathsf T_u$ associated with $u$ (roughly speaking,  $\mathsf T_u$ coincides with
 $\{|\nabla u|=1\}$ up to $\mm$-measure zero set) admits a disjoint family of unparameterized geodesics  $\{X_q\}_{q \in \mathfrak Q}$ such that 
\[
\mm(\mathsf T_u  \setminus \cup X_q)=0,
\]
and 
\[
\mm\restr{\mathsf T_u}=\int_{\mathfrak{Q}} \mm_q\, \d \mathfrak{q}(q),~~\mathfrak{q}(\mathfrak{Q})=1~~\text{and}~~\mm_q(X_q)=1~~\mathfrak{q}-\text{a.e.}~ q \in \mathfrak Q.
\]
Furthermore, for $\mathfrak{q}$-a.e. $q\in \mathfrak{Q}$, $\mm_q$ is a Radon measure with  $\mm_q \ll  \mathcal H^1 \restr{X_q}$  and
$(X_q, \d, \mm_q)$ satisfies $\mcp$.
\end{theorem}

\section{One dimensional models}
\subsection{One dimensional MCP densities}
Let  $h\in L^1(\R^+, \mathcal L^1)$ be a non-negative Borel function. It is  known (see e.g. Lemma 4.1 \cite{HM-MCP})  that $(\supp h, |\cdot|,  h \mathcal L^1)$ satisfies $\mcp$ condition if and only if  $h$  is a $\mcp$ density in the following sense
\begin{equation}\label{eq:mcp}
h(tx_1+(1-t)x_0) \geq \sigma^{(1-t)}_{K, N-1} (|x_1-x_0|)^{N-1} h(x_0)
\end{equation}
for all $x_0, x_1 \in \supp h$ and $t\in [0,1]$.

\begin{definition}
Given $K \in \R,  N>1$.  Denote by $D_{K,N}$ the Bonnet--Meyers diameter upper-bound:
\begin{equation} \label{eq:DKN}
D_{K,N} := \left\{\begin{array}{lll} \frac{\pi}{\sqrt{K / (N-1)}} & \text{if}~~ K > 0 \\
+\infty & \text{otherwise} 
\end{array} \right .  .
\end{equation}
 For any $D>0$, we define $ \mathcal F_{K, N, D}$ as the collection of $\mcp$ densities $h\in L^1(\R^+, \mathcal L^1)$ with $\supp h=[0, D\land D_{K, N}]$.

\end{definition}

For $\kappa \in \R$, we define the function $s_\kappa: [0, +\infty) \mapsto \R$ (on $[0, \pi/ \sqrt{\kappa})$ if $\kappa >0$)
\[
s_\kappa(\theta):=\left\{\begin{array}{lll}
(1/\sqrt {\kappa}) \sin (\sqrt \kappa \theta), &\text{if}~~ \kappa>0,\\
\theta, &\text{if}~~\kappa=0,\\
(1/\sqrt {-\kappa}) \sinh (\sqrt {-\kappa} \theta), &\text{if} ~~\kappa<0.
\end{array}
\right.
\]

It can be seen that  \eqref{eq:mcp}  is equivalent to
\begin{equation}\label{eq1}
\left (\frac {s_{K/(N-1)}(b-x_1)}{s_{K/(N-1)}(b-x_0)} \right )^{N-1} \leq \frac{h(x_1)}{h(x_0)} \leq \left (\frac {s_{K/(N-1)}(x_1-a)}{s_{K/(N-1)}(x_0-a)} \right )^{N-1}
\end{equation}
for all $ [x_0,  x_1] \subset [a, b] \subset  \supp h$.

Furthermore, we have the following characterization.
\begin{lemma}\label{lemma:mcp}
Given $D \leq D_{K, N}$, a  density $h$ is in $ \mathcal F_{K, N, D}$ if and only if
\begin{equation}\label{eq1.1}
\left (\frac {s_{K/(N-1)}(D-x_1)}{s_{K/(N-1)}(D-x_0)} \right )^{N-1} \leq \frac{h(x_1)}{h(x_0)} \leq \left (\frac {s_{K/(N-1)}(x_1)}{s_{K/(N-1)}(x_0)} \right )^{N-1}~~~\forall ~0\leq x_0 \leq x_1 \leq D.
\end{equation}

Furthermore,  $h\in \mathcal F_{K, N, D}$ if and only if  $\ln h$ is  $\mathcal L^1$-a.e.  differentiable  and
\[
 -h(x) \cot_{K, N, D}(D-x) \leq h'(x)\leq h(x) \cot_{K, N, D}(x), ~~\mathcal L^1-\text{a.e.}~x\in [0, D]
\]
where the function $\cot_{K, N, D}: [0, D] \mapsto \R$ is defined  by
\[
\cot_{K, N, D}(x):=\left\{\begin{array}{lll}
\sqrt{K(N-1)} \cot (\sqrt {\frac{K}{N-1}} x), &\text{if}~~ K>0,\\
({N-1})/x, &\text{if}~~K=0,\\
\sqrt{-K(N-1)} \coth(\sqrt {\frac{-K}{N-1}} x), &\text{if} ~~K<0.
\end{array}
\right.
\]
\end{lemma}
\begin{proof}
It can be checked that the function
\[
a  \mapsto \frac {s_{K/(N-1)}(x_1-a)}{s_{K/(N-1)}(x_0-a)}
\] is non-decreasing on $[0, x_0]$, and the function
\[
b  \mapsto \frac {s_{K/(N-1)}(b-x_1)}{s_{K/(N-1)}(b-x_0)}
\] is non-decreasing on $[x_1, D]$. Thus, \eqref{eq1.1} follows from  \eqref{eq1}.

Furthermore, for any $h\in \mathcal F_{K, N, D}$, it can be seen that \eqref{eq1.1} holds if and only if
\begin{equation}\label{eq:mcp-1}
x \mapsto \frac {\left ( {s_{K/(N-1)}(D-x)}\right )^{N-1}}{h(x)} ~~\text{is non-increasing},
\end{equation}
and 
\begin{equation}\label{eq:mcp-2}
x \mapsto \frac {\left ( {s_{K/(N-1)}(x)}\right )^{N-1}}{h(x)} ~~\text{is non-decreasing}.
\end{equation}
  From \eqref{eq1.1} we can see that  $\ln h$ is locally  Lipschitz, so $\ln h$ is differentiable almost everywhere.  So,  by \eqref{eq:mcp-1} and \eqref{eq:mcp-2}  we know \eqref{eq1.1} is equivalent to
$$
\Big ( \ln s^{N-1}_{K/(N-1)}(D-\cdot)  \Big ) ' \leq (\ln h)'=\frac  {h'}h \leq \Big ( \ln s^{N-1}_{K/(N-1)} \Big )'  ~~~~~\mathcal L^1-\text{a.e.~on}~ [0, D]
 $$ 
 which is the thesis.
\end{proof}

Notice  that the function
\[
[0, D] \ni x \mapsto \frac {s_{K/(N-1)}(D-x)}{s_{K/(N-1)}(x)}
\] is decreasing. 
By Lemma \ref{lemma:mcp} (or  \eqref{eq:mcp-1} and \eqref{eq:mcp-2} )  we immediately obtain the following rigidity result.

\begin{lemma}[One dimensional rigidity]\label{lemma:rigidity}
Denote  $h^1_{K, N, D}=\left ( {s_{K/(N-1)}(x)} \right )^{N-1}\restr{[0, D]}$ and $h^2_{K, N, D}= \left ( {s_{K/(N-1)}(D-x)} \right )^{N-1}\restr{[0,D]}$.
Then we have  $h^1_{K, N, D}, h^2_{K, N, D} \in \mathcal F_{K, N, D}$. Furthermore,  $h^1_{K, N, D}$  is the unique  $\mathcal F_{K, N, D}$ density (up to multiplicative constants) satisfying 
\[
h'(x)=h(x) \cot_{K, N, D}(x)
\]
and $h^2_{K, N, D}$  is the unique $\mathcal F_{K, N, D}$ density satisfying
\[
h'(x)=-h(x) \cot_{K, N, D}(D-x).
\]
\end{lemma}

\subsection{One dimensional $p$-Poincar\'e inequalities}
\begin{definition}
For $p \in (1, \infty)$ and $h \in \mathcal F_{K, N, D}$,  the $p$-spectral gap associated with $h$  is defined by
\begin{equation}\label{eq:pgap}
\lambda^{p, h}:= \inf \left \{\frac {\int |u'|^p h\,\d x}{\int |u|^p h \,\d x}: u \in \Lip \cap L^p, \int u|u|^{p-2}  h\,\d x=0, u\neq 0 \right \}.
\end{equation}
\end{definition}

\begin{definition}
Let  $K\in \R$,   $D>0$ and $N>1$. The optimal constant  $\lambda^{p}_{K, N, D}$ is defined as the infimum of  all $p$-spectral gaps associated with admissible densities, i.e.  $\lambda^{p}_{K, N, D}$  is given by
\[
\lambda^{p}_{K, N, D}:=\mathop{\inf}_{h \in \cup_{D' \leq D} \mathcal F_{K, N, D'}} \lambda^{p, h}.
\]
\end{definition}

\begin{proposition}\label{lemma1} Let $K\in \R$,   $D>0$ and $N>1$.  The function $D \mapsto \lambda^{p}_{K, N, D}$ is non-increasing, and
\begin{equation}\label{eq0:lm1}
\lambda^{p}_{K, N, D}=\mathop{\inf}_{h \in \cup_{D' \leq D} \mathcal F_{K, N, D'}\cap C^\infty}  \lambda^{p, h}.
\end{equation}

If $K\leq 0$, the map $D \mapsto \lambda^{p}_{K, N, D}$ is strictly decreasing, and
\begin{equation}\label{eq:lm1}
\lambda^{p}_{K, N, D}=\mathop{\inf}_{h \in \mathcal F_{K, N, D}\cap C^\infty} \lambda^{p, h}.
\end{equation}

\end{proposition}

\begin{proof}
By Lemma \ref{lemma:mcp} we know MCP densities are locally Lipschitz. Thus, using a standard mollifier  we can approximate $h$ uniformly  by smooth MCP densities. 
Then by a simple approximation argument (see e.g. Proposition 4.8 \cite{HM-MCP})  we can prove
\begin{eqnarray*}
\lambda^{p}_{K, N, D}= \mathop{\inf}_{h \in  \cup_{D' \leq D} \mathcal F_{K, N, D'}\cap C^\infty} \lambda^{p, h}.
\end{eqnarray*}

Let  $h \in \mathcal F_{K, N, D'}$ be a MCP density  for some $D'>0$, and $u$ be an admissible function in \eqref{eq:pgap}. Then $\bar h(x):=h(\frac {D'}{D} x) \in \mathcal F_{K', N, D}$  with $K'=\big (\frac {D'}{D} \big )^2 K$, and 
$\bar u(x):=u(\frac {D'}{D} x)$ is also an admissible function. By computation, we have  $\frac {\int |\bar u'|^p \bar h\,\d x}{\int |\bar u|^p \bar h\,\d x}=\big (\frac {D'}{D} \big )^p\frac {\int |u'|^p h\,\d x}{\int |u|^p h\,\d x}$. Therefore,  if $K \leq 0$ and $D' <D$, we have
\[
\mathop{\inf}_{h \in \mathcal F_{K, N, D}} \lambda^{p, h}   \leq  \mathop{\inf}_{h \in \mathcal F_{K', N, D}} \lambda^{p, h}  \leq  \left (\frac {D'}{D} \right  )^p \big (\mathop{\inf}_{h \in \mathcal F_{K, N, D'}} \lambda^{p, h}  \big ) < \mathop{\inf}_{h \in \mathcal F_{K, N, D'}} \lambda^{p, h} 
\]
and so
\[
\lambda^{p}_{K, N, D}  < \lambda^{p}_{K, N, D'}.
\]
 Then we obtain \eqref{eq:lm1}.
\end{proof}

\begin{remark}
The difference between the cases $K \leq 0$ and $K > 0$ was already observed in \cite{CS18} in the isoperimetric context and in \cite{HM-MCP} in the 2-Poincar\'e context.  It is known that the monotonicity property \eqref{eq:lm1} is \textbf{false} when $K> 0$.
\end{remark}

\bigskip

In order to study  the equation \eqref{eq1:lm1} in Theorem \ref{pgap}, we recall some basic facts about generalized trigonometric functions $\sin_p$ and $\cos_p$.
\begin{definition}
For $p \in (1, +\infty)$,  define $\pi_p$ by
\[
\pi_p:= \int_{-1}^1 \frac {\d t}{(1-|t|^p)^{\frac 1p}}=\frac {2\pi} {p \sin (\pi/p)}>0.
\]
The periodic  $C^1$ function $\sin_p: \R \mapsto [-1, 1]$ is defined on $[-\pi_p/2, 3\pi_p/2]$ by:
\begin{equation}
\left\{\begin{array}{ll}
t=\int^{\sin_p(t)}_0  \frac {\d s}{(1-|s|^p)^{\frac 1p}}~~~~~&\text{if}~t\in [-\frac {\pi_p} 2, \frac{\pi_p}2],\\
\sin_p(t)=\sin_p(\pi_p-t)~~~&\text{if}~t\in[\frac{\pi_p}2, \frac{3\pi_p}2].
\end{array}
\right.
\end{equation}
It can be seen that $\sin_p(0)=0$ and $\sin_p$ is strictly increasing on $ [-\frac {\pi_p} 2, \frac{\pi_p}2]$.
Define $\cos_p(t)=\ddt \sin_p(t)$, then we have the following generalized trigonometric identity
\[
|\sin_p(t)|^p+|\cos_p(t)|^p=1.
\]
\end{definition}

\begin{definition}
Let $h^i_{K, N, D}$, $i=1,2$ be $\mcp$ densities defined in Lemma \ref{lemma:rigidity}. Define   $h_{K, N, D}$ by
\begin{equation*}
h_{K, N, D}(x):=\left\{\begin{array}{ll||||}
h^1_{K, N, D}(x) ~~~~~~~~~~~\text{if}~x\in [\frac D 2, D]\\
h^2_{K, N, D}(x) ~~~~~~~~~~~\text{if}~x \in [0, \frac D 2].
\end{array}
\right.
\end{equation*}

Define $T_{K, N, D}$ by
\begin{equation*}
T_{K, N, D}:=\big (\ln h_{K, N, D} \big )'=\left\{\begin{array}{ll||||}
\cot_{K, N, D}(x) ~~~~~~~~~~~~\text{if}~x\in [\frac D 2, D]\\
-\cot_{K, N, D}(D-x) ~~~\text{if}~x \in [0, \frac D 2].
\end{array}
\right.
\end{equation*}
\end{definition}

By Lemma \ref{lemma:mcp}  we know   $h_{K, N, D}$ is a  $\mcp$ density. It can  be seen that   (c.f. Lemma 3.4  \cite{CS18}) $h_{K, N, D}$ does not satisfy any forms of CD  condition.

\begin{theorem}[One dimensional $p$-spectral gap]\label{pgap}
Let  $K\in \R$, $N>1$, $D>0$.  Denote by $\hat \lambda^{p}_{K, N, D}$  the  minimal $\lambda$ such that the following initial value problem  has a solution:
\begin{equation}\label{eq01:pgap}
\left\{\begin{array}{ll}
\varphi'=\Big (\frac \lambda {p-1} \Big )^{\frac 1p}+\frac 1{p-1} T _{K, N, D}\cos_p^{p-1}(\varphi)\sin_p(\varphi),\\
\varphi(0)=-\frac {\pi_p}2,~~\varphi(\frac D2)=0,~~ \varphi(D)=\frac {\pi_p}2.
\end{array}
\right.
\end{equation}
Then 
$\lambda^{p, h} \geq \hat \lambda^{p}_{K, N, D}$ for any $h \in \mathcal F_{K, N, D}$.

\end{theorem}

\begin{proof}

{\bf Step 1.} Firstly we will show the existence of $\hat \lambda^{p}_{K, N, D}$.

By Lemma \ref{lemma:mcp} we know $T_{K, N, D} \in C^\infty((0, \frac D2) \cup (\frac D2, D))$ and  $-\cot_{K, N, D}(D-\cdot) \leq T_{K, N, D} \leq  \cot_{K, N, D}$.  Denote  $T=T _{K, N, D}$, and denote by  $u=u^{T, \lambda}$  the (unique)  solution of the following equation:
\begin{equation}\label{eq3:pgap}
\left\{\begin{array}{ll}
\big (u' |u'|^{p-2} \big )'+T  u'|u'|^{p-2}+\lambda  u|u|^{p-2}=0,\\
u(\frac D 2)=0.
\end{array}
\right.
\end{equation}

Next  we will study the equation \eqref{eq3:pgap} using a version of the so-called Pf\"ufer transformation.
Define the functions $e=e^{T, \lambda}$ and $\varphi=\varphi^{T, \lambda}$ by:
\[
\alpha:= \Big (\frac \lambda {p-1} \Big )^{\frac 1p},~~~~\alpha u=e \sin_p(\varphi),~~~~~u'=e \cos_p (\varphi).
\]

By Lemma \ref{lemma2} we know  that $\varphi, e$ solve the following equation:
\begin{equation}\label{eq4:pgap}
\left\{\begin{array}{ll}
\varphi'=\alpha+\frac 1{p-1} T |\cos_p(\varphi) |^{p-2}\cos_p(\varphi) \sin_p(\varphi),\\
\ddt \ln e=\frac{e'}e=-\frac 1{p-1} T |\cos_p(\varphi)|^p.
\end{array}
\right.
\end{equation}

Consider the following initial valued problem on $(0, \frac D2) \cup (\frac D2, D)$.
\begin{equation}\label{eq4.1:pgap}
\left\{\begin{array}{ll}
\varphi'=\alpha+\frac 1{p-1} T |\cos_p(\varphi)|^{p-2}\cos_p(\varphi) \sin_p(\varphi),\\
\varphi(\frac D2)=0.
\end{array}
\right.
\end{equation}
By Cauchy's theorem we have the existence, uniqueness and continuous dependence on the parameters.  Fix an $\epsilon\in (0, \frac D2)$.  We can find  $\alpha=\alpha(\epsilon)>0$, such that  $\varphi'(x)>
\frac {\pi_p}{D-2\epsilon}>0$ for all $x\in (\epsilon, \frac D2)$. 
 So there exists $a_\alpha \in [0, \frac D2)$ such that   $\varphi(a_\alpha)= -\frac {\pi_p}2$.  Similarly, there is  $b_\alpha \in (\frac D2, D]$ such  that  $\varphi(b_\alpha)= \frac {\pi_p}2$.
 Conversely, assume   there is  $ \alpha>0$ such that  the following problem has a solution  $\varphi$ for some $a_\alpha \in [0, \frac D2)$ and $b_\alpha \in (\frac D2, D]$:
 \begin{equation}\label{eq5:pgap}
\left\{\begin{array}{ll}
\varphi'= \alpha+\frac 1{p-1}T |\cos_p(\varphi) |^{p-2}\cos_p(\varphi) \sin_p(\varphi),\\
\varphi(a_\alpha)= -\frac {\pi_p}2, \varphi(\frac D2)=0, \varphi(b_\alpha) = \frac {\pi_p}2.
\end{array}
\right.
\end{equation}
Then for any $\alpha'>\alpha$,  the following problem also  has a solution  for some $a'_\alpha\in (a_\alpha, \frac D2)$ and $b'_\alpha \in (\frac D2, b_\alpha)$
\begin{equation}\label{eq5.1:pgap} 
\left\{\begin{array}{ll}
\varphi'= \alpha'+\frac 1{p-1} T |\cos_p(\varphi) |^{p-2}\cos_p(\varphi) \sin_p(\varphi),\\
\varphi(a'_\alpha)= -\frac {\pi_p}2, \varphi(\frac D2)=0, \varphi(b'_\alpha) = \frac {\pi_p}2.
\end{array}
\right.
\end{equation}

 Therefore, by connectedness,  there is  a minimal $\bar \lambda \geq 0$ such that for any $\lambda> \bar \lambda$, there exist  $\varphi=\varphi^{T, \lambda}$,  $0\leq a^\lambda <\frac D2$ and $\frac D2 <b^\lambda \leq D$   such that 
 \begin{equation}\label{eq5.2:pgap}
\left\{\begin{array}{ll}
\varphi'= \Big (\frac \lambda {p-1} \Big )^{\frac 1p}+\frac 1{p-1} T\cos_p^{p-1}(\varphi)\sin_p(\varphi),\\
\varphi(a^\lambda)=-\frac {\pi_p}2, \varphi(\frac D2)=0, \varphi(b^\lambda)=\frac {\pi_p}2.
\end{array}
\right.
\end{equation}
By continuous dependence on the parameter $\lambda$,  we know  \eqref{eq5.2:pgap} has a solution $\varphi_\infty$ for  $\gamma=\bar \gamma$,  some $a^{\bar \lambda} \in  [0, \frac D2)$ and  $b^{\bar \lambda} \in ( \frac D2, D]$. In particular, $\bar \lambda >0$.

Since  $T(x)=-T(D-x)$ on $[0, \frac D2]$, by symmetry and minimality (or domain monotonicity)  of $\bar \lambda$, we have $a^{\bar \lambda}=0$ and $b^{\bar \lambda}=D$ (otherwise we can find a smaller $\lambda$).
In particular,   there is a minimal  $\hat \lambda^{p}_{K, N, D}$ such that the initial value problem  \eqref{eq01:pgap} has a solution $\varphi^{T _{K, N, D}, \hat \lambda^{p}_{K, N, D}}$.
\bigskip

{\bf Step 2.}
Given  $h \in \mathcal F_{K, N, D} \cap C^\infty$, we will show  that $ \hat \lambda^{p}_{K, N, D} \leq \lambda^{p, h}$.

 First of all,  by a standard variational argument  we can see that $\lambda^{p, h}$ is the smallest positive real number such that there exists a non-zero $u\in W^{1,p}([0, D], h  \mathcal L^1)$ solving  the following equation (in weak sense):
\begin{equation}\label{eq1:lm1}
\Delta_p^h u=-\lambda u|u|^{p-2}
\end{equation}
with Neumann boundary condition, where $ \Delta_p^h u$ is the weighted $p$-Laplacian  on $([0, D], |\cdot|, h \mathcal L^1)$:
$$
\Delta_p^h u=\Delta_p u+u'|u'|^{p-2} (\log h)'=\big (u' |u'|^{p-2} \big )'+u'|u'|^{p-2} \frac {h'}{h}.
$$
By regularity theory we know $u\in C^{1, \alpha} \cap W^{1,p}$ for some $\alpha>0$, and  $u\in C^{2, \alpha} $ if $u' \neq 0$.
Conversely, for any $u$ solving  the Neumann problem \eqref{eq1:lm1}, we have $ \int u|u|^{p-2} h\,\d x=0$ and $ {\int |u'|^ph\,\d x}=\lambda {\int |u|^p h\,\d x}$.

Assume by contradiction that  $\lambda^{p, h}<\hat \lambda^{p}_{K, N, D}$.  From the monotonicity argument in  {\bf Step 1},  we can see that there is  $\lambda < \hat \lambda^{p}_{K, N, D}$  such  that  the following equation  has a (monotone) solution
$\varphi= \varphi^{\frac{h'}h, \lambda} $:
 \begin{equation}\label{eq6:pgap}
\left\{\begin{array}{ll}
\varphi'= \Big (\frac \lambda {p-1} \Big )^{\frac 1p}+\frac 1{p-1} \frac{h'}h \cos_p^{p-1}(\varphi)\sin_p(\varphi),\\
\varphi(0)=-\frac {\pi_p}2, \varphi(D)=\frac {\pi_p}2,
\end{array}
\right.
\end{equation}
Without loss of generality (or by symmetry),  we may assume there is   $a'\in [\frac D2, D]$ such that $\varphi^{\frac{h'}h, \lambda}(a')=0$. Suppose  there is a point $x_0 \in [a', D)$  such that $\varphi^{\frac{h'}h, \lambda}(x_0)
 =\varphi^{T _{K, N, D}, \hat \lambda^{p}_{K, N, D}}(x_0)$. From  Lemma \ref{lemma:mcp} we know that $\frac{h'}h \leq T _{K, N, D}$.   So we know
  \[
\big (\varphi^{\frac{h'}h, \lambda}\big )' (x_0) < \big (\varphi^{T _{K, N, D}, \hat \lambda^{p}_{K, N, D}}\big )' (x_0).
 \]
Therefore, 
  $$\varphi^{\frac{h'}h, \lambda}(x)< \varphi^{T _{K, N, D}, \hat \lambda^{p}_{K, N, D}}(x) 
  $$ for all $x\in (a', D]$, which contradicts to the fact that  $\varphi^{\frac{h'}h, \lambda}(\frac D2)= \varphi^{T _{K, N, D}, \hat \lambda^{p}_{K, N, D}}(\frac D2)= \frac {\pi_p}2$.

 \end{proof}
 
 The following formulas has been used in  \cite{Valtorta12, NV14}. We give a proof for completeness.
 
 \begin{lemma}\label{lemma2}
 Let $e, \varphi, T$ be functions defined in the proof of Theorem \ref{pgap}.  Then we have
 \begin{equation}
\left\{\begin{array}{ll}
\varphi'=\alpha+\frac 1{p-1} T |\cos_p(\varphi) |^{p-2}\cos_p(\varphi) \sin_p(\varphi),\\
\ddt \ln e=\frac{e'}e=-\frac 1{p-1} T |\cos_p(\varphi)|^p.
\end{array}
\right.
\end{equation}
 \end{lemma}
 \begin{proof}
Firstly, we have
\begin{eqnarray*}
\big (u' |u'|^{p-2} \big )' &=& \big (e \cos_p (\varphi) |e \cos_p (\varphi)|^{p-2} \big )'\\
&=&  |e \cos_p (\varphi)|^{p-2} \big( e' \cos_p (\varphi)+e \sin''_p (\varphi)\varphi'\big)\\
&&+ e \cos_p (\varphi) (p-2)e\cos_p (\varphi) |e \cos_p (\varphi)|^{p-4} \big( e' \cos_p (\varphi)+e \sin''_p (\varphi)\varphi'\big)\\
&=& |e \cos_p (\varphi)|^{p-2}(p-1) \big( e' \cos_p(\varphi)+e \sin''_p(\varphi)\varphi' \big).
\end{eqnarray*}
Combining with \eqref{eq3:pgap} we obtain
\begin{eqnarray*}
&&|e \cos_p (\varphi)|^{p-2}\big( e' \cos_p(\varphi)+e \sin''_p(\varphi)\varphi' \big)\sin_p(\varphi)\\
&+& \frac 1{p-1} T e \cos_p(\varphi) \sin_p(\varphi) | e \cos_p(\varphi)|^{p-2}+\frac \lambda{p-1} \alpha^{1-p} e^{ p-1}|  \sin_p(\varphi)|^{p}=0. 
\end{eqnarray*}

  Differentiating the  equation $\alpha u=e \sin_p(\varphi)$ and  substituting $u'$ by $e \cos_p (\varphi)$,  we get
\[
\alpha e \cos_p (\varphi)=e' \sin_p(\varphi)+e \cos_p(\varphi)\varphi'.
\]
Differentiating the identity $|\sin_p(t)|^p+|\cos_p(t)|^p=1$ we also have
\[
|\sin_p(t)|^{p-2}\sin_p(t) \cos_p(t)+|\cos_p(t)|^{p-2}\cos_p(t) \sin''_p(t)=0.
\]
Therefore,
 \begin{eqnarray*}
&&|e \cos_p (\varphi)|^{p-2}\big( e' \cos_p(\varphi)+e \sin''_p(\varphi)\varphi' \big)\sin_p(\varphi)\\&=&
|e \cos_p (\varphi)|^{p-2}\big( \alpha e \cos^2_p (\varphi)-e  \cos^2_p (\varphi) \varphi'+e \sin''_p(\varphi)\sin_p(\varphi)\varphi' \big)\\
&=& \alpha e^{p-1} (\varphi)  | \cos_p (\varphi)|^{p}-e^{p-1}\big( | \cos_p (\varphi)|^{p}+| \sin_p (\varphi)|^{p}   \big)\varphi'\\
&=&   \alpha e^{p-1} (\varphi)  | \cos_p (\varphi)|^{p}-e^{p-1}\varphi'.
\end{eqnarray*}
Combining the results above, we prove the lemma.
 
 \end{proof}

 Combining Proposition \ref{lemma1} and Theorem \ref{pgap}, we get the following corollary immediately.  
 
 \begin{corollary}\label{coro}
 We have the following sharp $p$-spectral gap estimates for one dimensional models:
 \[
\lambda^{p}_{K, N, D}= \left\{\begin{array}{lll}  
\hat \lambda^{p}_{K, N, D} & \text{if}~~ K \leq 0 \\
\inf_{D' \in (0,\min(D,D_{K,N})]} \hat \lambda^{p}_{K, N, D'} & \text{if}~~ K > 0 \\
\end{array}
\right . 
\]
 \end{corollary}

\bigskip

\begin{theorem}[One dimensional rigidity]\label{th-rigidity}
Given $K\leq 0$, $N>1$ and $D>0$. If $\lambda^{p, h}=\hat \lambda^{p}_{K, N, D}$ for some $h \in \mathcal F_{K, N, D}$. Then $h=h_{K, N, D}$  up to a multiplicative constant.
\end{theorem}
\begin{proof}

Assume  $\lambda^{p, h}=\hat \lambda^{p}_{K, N, D}$ for some $h \in \mathcal F_{K, N, D}$. Then there is $h_n \in \mathcal F_{K, N, D}\cap C^\infty$ with  $h_n \to h$ uniformly, 
and a decreasing sequence  $(\lambda^{p, h_n})$ with $\lambda^{p, h_n} \to \hat \lambda^{p}_{K, N, D}$,  such that  $\varphi_n=\varphi^{\frac{h'_n}{h_n}, \lambda^{p, h_n}}$ solves the following equation:
 \begin{equation}\label{eq7:pgap}
\left\{\begin{array}{ll}
\varphi'_n= \Big (\frac {\lambda^{p, h_n}} {p-1} \Big )^{\frac 1p}+\frac 1{p-1} \frac{h'_n}{h_n} \cos_p^{p-1}(\varphi_n)\sin_p(\varphi_n),\\
\varphi_n(0)=-\frac {\pi_p}2,  \varphi_n(D)=\frac {\pi_p}2.
\end{array}
\right.
\end{equation}
From Lemma \ref{eq-MCP}  we know  that $\{\varphi'_n\}_n$  and $\{\varphi_n\}_n$ are uniformly bounded.  By Arzel\`a-Ascoli theorem we may assume
$\varphi_n \to \varphi_\infty$  uniformly for some Lipschitz function $\varphi_\infty$.

 By   minimality of $\hat \lambda^{p}_{K, N, D}$ and symmetry, we can see that   $\lmt{n}\infty \varphi_n^{-1}(t)$ exists for any $t\in  [-\frac{\pi_p}2, \frac{\pi_p}2]$ and 
 $$\lmt{n}\infty \varphi_n^{-1}= \big(\varphi^{T _{K, N, D}, \hat \lambda^{p}_{K, N, D}}\big )^{-1}.
 $$
 In fact, assume by contradiction that $\lmt{n}\infty \varphi_n^{-1} (t) \neq  \big(\varphi^{T _{K, N, D}, \hat \lambda^{p}_{K, N, D}}\big )^{-1}(t)$ for some $t\in (-\frac{\pi_p}2, \frac{\pi_p}2)$. By symmetry we may assume  there are  $N_1\in \N$ and $\delta>0$, such that 
 $\delta_n:= \big(\varphi^{T _{K, N, D}, \hat \lambda^{p}_{K, N, D}}\big )^{-1}(t)-\varphi_n^{-1} (t)\geq \delta$ for all  $n\geq N_1$.
Define a  $\mcp$ density $\bar h_n$ by
\begin{equation*}
\bar h_{n}(x):=\left\{\begin{array}{ll||||}
h_n(x) ~~~~~~~~~~~~~~~~~~~~~~~~~~~~~~~~~~~~~~~~\text{if}~x\in [0, \varphi_n^{-1} (t)],\\
\frac{h_n( \varphi_n^{-1} (t))}{h_{K, N, D}( \varphi_n^{-1} (t)+\delta_n)}h_{K, N, D}(x+\delta_n) ~~~~~~~~\text{if}~x \in [\varphi_n^{-1} (t), D-\delta_n].
\end{array}
\right.
\end{equation*}
 Then $\bar \varphi_n=\varphi^{\frac{\bar h'_n}{\bar h_n}, \lambda^{p, h_n}}$ satisfies $(\bar \varphi_n)^{-1}(\frac{\pi_p}2)<D-\frac\delta 2$ for $n$ large enough, 
which contradicts to  Proposition \ref{lemma1} and the minimality of $\hat \lambda^{p}_{K, N, D}$.
 
 In conclusion, $\varphi_\infty=\varphi^{T _{K, N, D}, \hat \lambda^{p}_{K, N, D}}$ and we have  $\varphi_n \to \varphi^{T _{K, N, D}, \hat \lambda^{p}_{K, N, D}}$ uniformly.

Then we get
\begin{eqnarray*}
\frac {\pi_p} 2&= & \varphi_n( \varphi_n^{-1}(0))-\varphi_n(0)\\
&=&\lmt{n}{\infty} \int_0^{\varphi_n^{-1}(0) \land \frac D2}  \Big (\frac {\lambda^{p, h_n}} {p-1} \Big )^{\frac 1p}+\frac 1{p-1} \frac{h'_n}{h_n} \cos_p^{p-1}(\varphi_n)\sin_p(\varphi_n) \,\d x\\
&\leq & \lmt{n}{\infty} \int_0^{\varphi_n^{-1}(0) \land \frac D2} \Big (\frac {\lambda^{p, h_n}} {p-1} \Big )^{\frac 1p}+ \frac 1{p-1} T_{K, N, D}\cos_p^{p-1}(\varphi_n)\sin_p(\varphi_n) \,\d x\\
&=&  \int_0^{ \frac D2} \Big (\frac {\hat \lambda^{p}_{K, N, D}} {p-1} \Big )^{\frac 1p}+ \frac 1{p-1} T_{K, N, D}\cos_p^{p-1}(  \varphi^{T _{K, N, D}, \hat \lambda^{p}_{K, N, D}})\sin_p(\varphi^{T _{K, N, D}, \hat \lambda^{p}_{K, N, D}}) \,\d x\\
&=& \frac {\pi_p} 2.
\end{eqnarray*}
Therefore,
$$
(\ln h_n)'=\frac{h'_n}{h_n}  \to T_{K,N,D}=\frac {h'_{K,N,D}}{h_{K,N,D}}
$$ in $L^1([0, \frac D2],  \cos_p^{p-1}(  \varphi^{T_{K, N, D}, \hat \lambda^{p}_{K, N, D}})\sin_p(\varphi^{T _{K, N, D}, \hat \lambda^{p}_{K, N, D}}) \mathcal L^1)$. By symmetry, we can see that $(\ln h_n)' \to (\ln h_{K,N,D})'$ in  $L^1([0, D], \mathcal L^1)$. Hence   $h=h_{K, N, D}$  up to a multiplicative constant.

\end{proof}

\section{$p$-spectral gap}

\subsection{Sharp $p$-spectral gap estimates}
Using  standard  localization argument (c.f. Theorem 1.1 \cite{HM-MCP},  Theorem 4.4  \cite{CM-S}), we can  prove  the  sharp $p$-Poincar\'e inequality  with  one dimensional results.   

\begin{theorem}[The sharp $p$-spectral gap under $\mcp$]\label{th1}
Let $\ms$ be an essentially non-branching metric measure space satisfying $\mcp$  for some $K \in \R, N\in (1, \infty)$ and ${\rm diam}(X) \leq D$.    For any $p>1$,  define $\lambda^p_{\ms}$ as the optimal constant in $p$-Poincar\'e inequality on $\ms$:
\[
\lambda^{p}_{\ms} := \inf \left \{\frac {\int |\nabla  f|^p \,\d \mm}{\int |f|^p  \,\d \mm}: f \in \Lip \cap L^p, \int f|f|^{p-2}  \,\d \mm=0, f\neq 0 \right \}.
\]

Then we have the following sharp estimate
\[
\lambda^{p}_{\ms} \geq \lambda^{p}_{K, N, D}  = \left\{\begin{array}{lll}  
\hat \lambda^{p}_{K, N, D} & \text{if}~~ K \leq 0 \\
\inf_{D' \in (0,\min(D,D_{K,N})]} \hat \lambda^{p}_{K, N, D'} & \text{if}~~ K > 0.\\
\end{array}
\right. 
\]
\end{theorem}
\begin{proof}
Let $\bar f=f|f|^{p-2} $ be a Lipschitz function with $\int \bar f=0$.   Let  $\bar f^\pm$ denote the positive and the negative parts of $\bar f$ respectively. Then we have  $\int \bar f^+=-\int \bar f^-$. Consider  the $L^1$-optimal transport problem from $\mu_0:=\bar f^+\mm$ to $\mu_1:= -\bar f^-\mm$. By Theorem \ref{prop-l1},  there exists  a family of disjoint unparameterized geodesics $\{X_q\}_{q \in \mathfrak Q}$ of length at most $D$, such that 
\[
\mm(X \setminus \cup X_q)=0,~~~\mm=\int_{\mathfrak{Q}} \mm_q\, \d \mathfrak{q}(q)
\]
where $\mm_q=h_q \mathcal H^1 \restr{X_q}$  for some $h_q \in \mathcal F_{K, N, D'_q}$ with $D'_q\leq D$, $\mm_q(X_q)=\mm(X)$ and
\[
\int \bar f h_q \, \d \mathcal H^1 \restr{X_q}=0
\]
for $\mathfrak{q}$-a.e. $q \in \mathfrak{Q}$.

Denote $f_q=f \restr{X_q}$. By definition we obtain
\[
\int |f_q'| ^p h_q \, \d \mathcal H^1 \restr{X_q} \geq   \lambda^{p, h_q}\int |f_q|^p h_q \, \d \mathcal H^1 \restr{X_q} \geq  \lambda^{p}_{K, N, D}\int |f_q|^p h_q \, \d \mathcal H^1 \restr{X_q}.
\]
Notice that $|f_q'| \leq |\nabla f|$. Thus, we have
\begin{eqnarray*}
\lambda^{p}_{K, N, D}  \int  | f|^p \,\d \mm &=&   \lambda^{p}_{K, N, D}   \int_{\mathfrak{Q}} \int_{X_q} | f_q|^p \mm_q\, \d \mathfrak{q}(q)  \\
&\leq&  \int_{\mathfrak{Q}} \int_{X_q} | f'_q|^p \mm_q\, \d \mathfrak{q}(q) \\
&=&  \int  |\nabla  f|^p \,\d \mm.
\end{eqnarray*}
Combining with Corollary \ref{coro} we prove the theorem.
\end{proof}

\subsection{Rigidity for  $p$-spectral gap}

In this part,  we will study the  rigidity for $p$-spectral gap under the measure contraction property.  We adopt the notation $|\D f|$ to denote the weak upper gradient of a Sobolev function $f$. We refer the readers to \cite{AGS-C} and \cite{G-O} for details about  Sobolev space theory and calculus on metric measure spaces.

\begin{theorem}[Rigidity for $p$-spectral gap]\label{th2}
Let $\ms$ be an essentially non-branching metric measure space satisfying $\mcp$  for some $K \leq 0, N\in (1, \infty)$ and ${\rm diam}(X) \leq D$. Assume there is 
a non-zero Sobolev function $f\in W^{1,p}\ms$ with  $\int f|f|^{p-2}  \,\d \mm=0$ such that 
\[
{\int |\D  f|^p \,\d \mm}-\hat \lambda^{p}_{K, N, D}{\int |f|^p  \,\d \mm}=0.
\] 
Then ${\rm diam}(X)=D$  and there are disjoint unparameterized geodesics $\{X_q\}_{q \in \mathfrak Q}$ of length  $D$ such that $\mm(X \setminus \cup X_q)=0$.  Moreover, $\mm$ has the following representation
\[
\mm=\int_\mathfrak{Q} h_q \, \d \mathcal H^1 \restr{X_q} \d \mathfrak{q}(q),
\]
where  $\frac{h'_q}{h_q}=T^p_{K, N, D}$ for $\mathfrak{q}$-a.e. $q \in \mathfrak{Q}$.
\end{theorem}

\begin{proof}
  Similar to the proof of Theorem \ref{th1}, we can find a 
measure decomposition  associated with $\bar f:= f|f|^{p-2}$,  such that 
\[
\mm(X \setminus \cup X_q)=0,~~~\mm=\int_{\mathfrak{Q}} \mm_q\, \d \mathfrak{q}(q)
\]
where $\mm_q=h_q \mathcal H^1 \restr{X_q}$  for some $h_q \in \mathcal F_{K, N, D'_q}$ with $D'_q\leq D$,  $\mm_q(X_q)=\mm(X)$ and
\[
\int \bar f h_q \, \d \mathcal H^1 \restr{X_q}=0
\]
for $\mathfrak{q}$-a.e. $q \in \mathfrak{Q}$.

By Theorem 7.3  \cite{AGS-D} we know  $f_q:=f \restr{X_q} \in W^{1,q}(X_q)$ and $|\D f_q| \leq |\D f|$. Then from the proof of Theorem \ref{th1} we can see 
that $ \lambda^{p, h_q}=\hat \lambda^{p}_{K, N, D}$ for $\mathfrak{q}$-a.e. $q \in \mathfrak{Q}$. By Proposition \ref{lemma1} we know that the function  $D\mapsto \hat \lambda^{p}_{K, N, D}$ is strictly decreasing, so  $D'_q= D$ and ${\rm diam}(X)=D$.
Finally,   by Theorem  \ref{th-rigidity} we  know that  $\frac{h'_q}{h_q}=T^p_{K, N, D}$.
\end{proof}

\def\cprime{$'$}

\end{document}